\documentclass[10pt]{amsart}
\usepackage[a4paper,margin=1.5in,footskip=0.25in]{geometry}

\usepackage[utf8]{inputenc}
\usepackage[english]{babel}
\usepackage{amsthm}
\usepackage{amsmath}
\usepackage{tikz-cd}
\usepackage{amssymb}
\usepackage{comment}
\usepackage{eucal}
\usepackage{hyperref}
\usepackage{array}
\usepackage[new]{old-arrows}
\usepackage[capitalise,noabbrev]{cleveref}
\usepackage{graphicx}
\usepackage{authblk}
\usepackage{amsaddr}
\usepackage{xcolor}
\hypersetup{
    colorlinks,
    linkcolor={red!50!black},
    citecolor={blue!50!black},
    urlcolor={blue!80!black}
}

\tikzcdset{arrow style=math font}

\newtheorem{theorem}{Theorem}[section]
\newtheorem{proposition}[theorem]{Proposition}
\newtheorem{lemma}[theorem]{Lemma}
\newtheorem{corollary}[theorem]{Corollary}

\theoremstyle{definition}
\newtheorem{definition}[theorem]{Definition}
\newtheorem{remark}[theorem]{Remark}
\newtheorem{example}[theorem]{Example}
\newtheorem{notation}[theorem]{Notation}
\newtheorem{step}{\normalfont\scshape Step}

\newcommand{\C}{\CMcal{C}}
\newcommand{\D}{\CMcal{D}}
\newcommand{\X}{\CMcal{X}}
\newcommand{\Y}{\CMcal{Y}}
\newcommand{\PSh}{\CMcal{P}}
\newcommand{\Z}{\CMcal{Z}}
\newcommand{\E}{\CMcal{E}}
\newcommand{\Sp}{\CMcal{S}}
\newcommand{\Lie}{\mathsf{Lie}}
\newcommand{\sk}{\mathsf{sk}}
\newcommand{\Grp}{\mathsf{Grp}}
\newcommand{\Mon}{\mathsf{Mon}}
\newcommand{\DDelta}{\mathbf{\Delta}}
\newcommand{\sSet}{\mathsf{sSet}}
\newcommand{\s}{\mathsf{s}}
\newcommand{\op}{\mathsf{op}}
\newcommand{\Set}{\mathsf{Set}}
\newcommand{\weak}{\overset{\sim}{\to}}
\newcommand{\longweak}{\overset{\sim}{\longrightarrow}}
\newcommand{\XX}{\underline{X}}
\newcommand{\colim}{\mathsf{colim}}

\newcommand{\conn}{\mathsf{conn}}
\newcommand{\id}{\mathsf{id}}

\newcommand{\ad}{\mathsf{ad}}
\newcommand{\mmin}{\mathsf{min}}
\newcommand{\fib}{\mathsf{fib}}
\newcommand{\cof}{\mathsf{cof}}
\newcommand{\Map}{\mathsf{Map}}
\newcommand{\N}{\mathbb N}
\newcommand{\EE}{\mathbb E}
\newcommand{\Alg}{\mathsf{Alg}}
\newcommand{\Fun}{\mathsf{Fun}}
\newcommand{\Cat}{\mathsf{Cat}}

\title{Hilton-Milnor's theorem in $\infty$-topoi.}
\author{\normalfont\scshape Samuel Lavenir}
\address{\normalfont EPFL, Lausanne}

\begin{document}

\maketitle

\begin{abstract}
    In this note we show that the classical theorem of Hilton-Milnor on finite wedges of suspension spaces remains valid in any $\infty$-topos. Our result relies on a version of James' splitting proved in \cite{EHP} and uses only basic constructions native to any model of $\infty$-categories. 
\end{abstract}

\tableofcontents

\section{Introduction}
The aim of this article is to give an $\infty$-categorical treatment of the decomposition theorem for wedges of suspensions first proved by Hilton \cite{hiltonoriginal} in the case of spheres, and generalized by Milnor \cite{milnor_fk} to arbitrary suspension spaces. The classical form of this theorem states that given pointed connected CW-complexes $X_1, \cdots, X_n$, there is a canonical homotopy equivalence $$
\prod_{w\in B(n)} \Omega \Sigma w(X_1, \cdots, X_n) \longweak \Omega \Sigma (X_1 \vee \cdots \vee X_n)
$$
where $B(n)$ denotes a basis for the free Lie algebra on $n$-generators. Several remarks are in order: first, given a word $w$ in this free Lie algebra, $w(X_1, \cdots, X_n)$ denotes the pointed CW-complex obtained by forming the smash product of the $X_i$ in the order indicated by $w$. Secondly, the topology on the product space is \textit{not} the product topology, but instead the \textit{weak topology} with respect to the filtration by all finite products. Lastly, it should be noted that the assumption that the $X_i$ are connected is essential and cannot be weakened.

Our main result is that the same statement holds true in any $\infty$-topos, in the sense of \cite{HTT,rezktopos, TV}. Aside from giving extended generality to the original statement, this new formulation makes the argument conceptually simpler; compare with traditional proofs \cite[Theorem~A]{hiltonoriginal}, \cite[Theorem~4.3.3]{neisendorfer_2010}, \cite[Theorem~6.6]{bookwhitehead}. Our proof involves only basic categorical manipulations, as well as the combinatorics of Hall words which we review in \Cref{hallsection} and an $\infty$-categorical version of James' splitting \cite{EHP}. The homotopical input is entirely hidden in the higher coherences involved in the formation of limits and colimits in $\infty$-categories. 

It should be noted that the contribution of this paper lies merely in the proof of Hilton-Milnor's theorem using higher categorical techniques, rather than its statement. Indeed, the result could be simply derived from the usual theorem in the $\infty$-category of spaces by the fact that an $\infty$-topos $\E$ admits a presentation
\[
\begin{tikzcd}
L: \PSh (\C) \ar[r, shift left] &  \E \ar[l, shift left, hook] :i
\end{tikzcd}
\]
as a left-exact localization of a presheaf $\infty$-category. Left-exactness of the left adjoint $L$ implies that it commutes with all the operations $\Omega, \Sigma, \wedge$ involved in the theorem, as well as with the weak product. It follows that the statement of the theorem in an arbitrary $\infty$-topos follows from the one in the $\infty$-category $\Sp$ of spaces. In this work, we do not assume Hilton-Milnor's theorem holds in spaces and prove it from first principles. The elementary form of our proof makes it a suitable reference for the reader familiar to $\infty$-categories and willing to take a modern approach on Hilton-Milnor's theorem.
\medskip

\noindent

{\bf Acknowledgments.} We are grateful to Hadrian Heine for pointing out that the statement of Hilton-Milnor's theorem in an arbitrary  $\infty$-topos can be deduced from the corresponding one in spaces. We warmly thank Nima Rasekh and Jérôme Scherer for precious advice and comments on earlier versions of this note.

\section{Categorical preliminaries}
\subsection{Notational conventions}
In this section we fix some conventions about our use of the language of $\infty$-categories. By an $\infty$-category we mean an $(\infty,1)$-category, ie. a weak $\omega$-category in which all $i$-morphisms are (weakly) invertible for $i>1$. Although we do not appeal to any such in the sequel, several equivalent models of $\infty$-categories are in use, amongst which the most common include  quasicategories, simplicial categories, complete Segal spaces, Segal categories and relative categories.  The sense in which all these models are equivalent is discussed extensively in the litterature, see for example \cite{bergnermodels} and \cite{joyaltierney}. All of these models support well-defined notions of composition of maps, limits and colimits, functors, natural transformations, adjunctions, Kan extensions and so on.

In this paper we will be almost exclusively working inside a fixed $\infty$-category. Throughout we will be using the following abuse of terminology : when we say that a diagram \textit{commutes} one should understand that there are homotopies filling all simplices of the diagram, and that these homotopies are \textit{specified} as part of the data of the diagram. We never explicitly mention these homotopies since in all cases of interest to us they are either uniquely determined (as in the case of compositions), or provided by a universal property (as in the case of (co)limits). When we speak of limits and colimits, we really mean their homotopical analogues (any other notion would be meaningless in this context). Similarly, we say that an arrow is an \textit{equivalence} (or that it is \textit{invertible}) if it is weakly invertible, and the adjective \textit{unique} means `up to a contractible space of choices'. For us a \textit{space} is really an $\infty$-groupoid, and the $\infty$-category of these is denoted by $\Sp$. When necessary, we implicitly view 1-categories as $\infty$-categories through the nerve embedding. 

If $\C$ is an $\infty$-category, we write $h\C$ for its \textit{homotopy category}, which is a 1-category. There is a projection functor $\C \longrightarrow h\C$ which is universal amongst functors from $\C$ to a 1-category. The (discrete) mapping spaces of $h\C$ are denoted $[X,Y]$, while those of $h\C_*$ are denoted $[X,Y]_*$. If $X\in \C$ is an object of $\C$ we write $\id_X$ for the identity map on $X$, or even just $\id$ when no confusion can arise. Also, we denote by $\PSh (\C)$ the $\infty$-category $\Fun(\C^\op, \Sp)$ of \textit{presheaves} (with values in $\infty$-groupoids) and by $\C^\to$ the arrow category $\Fun(\Delta^1, \C)$.

Let $\C$ be an $\infty$-category which has finite limits and colimits. We denote by $*\in \C$ a terminal object and write $\C_*=\C_{*/}$ for the category of pointed objects of $\C$, which is the coslice of $\C$ under $*$. The forgetful functor $\C_* \longrightarrow \C$ has a left adjoint $(-)_+$ which can be described on objects by $X_+=X \coprod *$. We write $\Map_{\C}(X,Y)$ for the space of morphisms between objects $X$ and $Y$ of $\C$, or simply $\Map(X,Y)$ when no confusion can arise. We also write $\Map_*(X,Y)$ for the space of maps in the pointed category $\C_*$. The space of pointed maps between objects $X,Y \in \C_*$ fits in the following pullback square of spaces
$$
\begin{tikzcd}
   \Map_*(X,Y)\arrow[dr, phantom, "\scalebox{1}{$\lrcorner$}" , very near start, color=black] \ar[r]\ar[d] & \Map(X,Y) \ar[d] \\
   * \ar[r] & \Map(*,Y)
\end{tikzcd}
$$
where the bottom horizontal map corresponds to the base point $* \longrightarrow Y$.

The coproduct of two objects $X,Y\in \C_*$ is denoted by a wedge $X\vee Y$ in order to distinguish it from the coproduct in the unpointed category which we denote by $X\sqcup Y$. Similarly, the coproduct of a family $\{X_i\}_{i\in I}$ of objects in $\C_*$ will be denoted by $\bigvee_{i\in I} X_i$.

Finally, when facing an adjoint situation 
\[
\begin{tikzcd}
L: \C \ar[r, shift left] &  \D \ar[l, shift left] :R
\end{tikzcd}
\]
the left adjoint is always depicted on top.

\subsection{Background on $\infty$-categories}
In this section, we introduce some terminology for working with $\infty$-categories, following \cite{HTT}. We also include several basic results about (co)limits and group objects in $\infty$-categories.

An $\infty$-category $\C$ is pointed if it has a zero object, which we always denote by $*$. If $\C$ is an $\infty$-category with a terminal object, then $\C_*$ is a pointed $\infty$-category. 

A sequence $$
\cdots \longrightarrow X_{n-1}\longrightarrow X_n \longrightarrow X_{n+1} \longrightarrow \cdots
$$ of maps in a pointed $\infty$-category $\C_*$ is said to be a \textit{fiber sequence} (resp. \textit{cofiber sequence}) if the squares $$
\begin{tikzcd}
X_{n-1} \ar[r]\ar[d] & X_n \ar[d] \\
* \ar[r] & X_{n+1} 
\end{tikzcd}
$$
are Cartesian (resp. coCartesian) for every $n$. Given a map $f:X\longrightarrow Y$ in a pointed $\infty$-category $\C$, we write $\fib(f)$ and $\cof(f)$ for its fiber and cofiber, so that $\fib(f) \longrightarrow X \longrightarrow Y$ is a fiber sequence and $ X \longrightarrow Y \longrightarrow \cof(f)$ is a cofiber sequence. These assemble to give functors $\fib, \cof: \C^\to \longrightarrow \C$ on the arrow category of $\C$. 


\begin{lemma}\cite[Lemma~4.4.2.1]{HTT}\label{pasting}
Let $\C$ be an $\infty$-category with pushouts. Suppose that $$
\begin{tikzcd}
 X \arrow[dr, phantom, "\scalebox{1}{$\ulcorner$}" , very near end, color=black]\ar[r]\ar[d] & Y \ar[d] \ar[r] & Z \ar[d] \\
X' \ar[r] & Y' \ar[r] & Z'  
\end{tikzcd}
$$
is a commutative diagram in $\C$ where the left square is coCartesian. 
Then the right square is coCartesian if and only if the total rectangle is coCartesian.
\end{lemma}

\begin{proposition}\cite[Lemma~5.5.2.3]{HTT}\label{fubini}
    Let $\C$ be an $\infty$-category with pushouts. Suppose that $$
\begin{tikzcd}
 X  &  X' \ar[l] \ar[r] &  X''   \\
 Z \ar[d]\ar[u] &  Z' \ar[r]\ar[d]\ar[l]\ar[u] & Z'' \ar[d]\ar[u] \\
Y &  Y' \ar[l] \ar[r] &  Y''   
\end{tikzcd}
$$
is a commutative diagram in $\C$. Then the colimit of the above diagram exists and is equivalent to the colimit of both following diagrams : 
$$
\begin{aligned}
 X \coprod_{Z} Y \longleftarrow X' \coprod_{Z'} Y' \longrightarrow X'' \coprod_{Z''} Y''  \\
 X \coprod_{X'} X'' \longleftarrow Y \coprod_{Y'} Y'' \longrightarrow Z \coprod_{Z'} Z''.
\end{aligned}
$$
\end{proposition}

\begin{definition}
    An $\infty$-category $\C$ is said to have \textit{universal pushouts} if \begin{enumerate}
        \item all pullbacks and pushouts exist in $\C$
        \item for any $f:X\longrightarrow Y$ in $\C$, the pullback functor $f^*: \C_{/Y} \longrightarrow \C_{/X}$ preserves pushouts.
    \end{enumerate}
\end{definition}
\begin{remark}
    Since colimits in $\C_{/X}$ are computed in the underlying category $\C$ after forgetting the maps to $X$, an $\infty$-category $\C$ with all pushouts and pullbacks has universal pushouts if and only if pullbacks and pushouts commute in the following sense: given a pushout square in $\C$ as on the left and any map $f:f^*D\longrightarrow D$, the resulting square on the right, obtained by pulling back along $f$
    $$\begin{tikzcd}
    A \arrow[dr, phantom, "\scalebox{1}{$\ulcorner$}" , very near end, color=black]\ar[d]\ar[r] & B \ar[d] \\
    C \ar[r] & D
    \end{tikzcd}
    \:\:\:\:\:\:\:\:\:\:\:\:\:\:\:\:\:\:\:\:\:\:\:\:\:
    \begin{tikzcd}
        f^*A \arrow[dr, phantom, "\scalebox{1}{$\ulcorner$}" , very near end, color=black]\ar[d]\ar[r] & f^*B \ar[d] \\
        f^*C \ar[r] & f^*D
    \end{tikzcd}
    $$
    is again a pushout in $\C$. In symbols this gives $$
    f^*(B\sqcup_A C)\simeq f^*B \sqcup_{f^*A} f^*C.
    $$ 
This is exactly the \textit{first cube lemma} which was proved by M. Mather in \cite{mather_1976} for homotopy colimits of spaces. Using the identification \cite[Theorem~4.2.4.1]{HTT} of colimits in an $\infty$-category $\C$ with homotopy colimits in a presentation of $\C$ by a fibrant simplicial category, this can be interpreted as a proof that the $\infty$-category $\Sp$ of spaces has universal pushouts. See \cite[Lemma~2.6]{EHP} for a complete proof of the fact that $\C$ has universal pushouts if and only if it satisfies Mather's first cube lemma.
\end{remark}

The following `descent' statement is a typical case where universal pushouts come useful in practice. It is an analogue of Ganea's result for `fibers of cofibers' \cite{Ganea}. 

\begin{proposition}\label{descent}
    Let $\C$ be an $\infty$-category with universal pushouts, and let $$
    \begin{tikzcd}
    X \ar[d] & Z \ar[d] \ar[r]\ar[l] & Y \ar[d] \\
    T \ar[r, equals] & T \ar[r, equals] & T  
    \end{tikzcd}
    $$
    be a diagram in $\C_*$. Then there is a canonical equivalence $$
    \fib\bigg(X\coprod_Z Y \to T\bigg) \simeq \fib(X \to T) \coprod_{\fib(Z \to T)} \fib(Y \to T).
    $$
\end{proposition}
\begin{proof}
    This is just an incarnation of the cube lemma stated above, in the case where $f$ is the map $*\to T$. The conclusion follows from the fact that pushouts and pullbacks in the pointed category $\C_*$ can be computed in $\C$. 
\end{proof}

\subsection{Presentability and topoi}

In this section we recast basic definitions and facts about presentable $\infty$-categories and $\infty$-toposes. 

An $\infty$-category $\E$ is \textit{presentable} if it arises as an accessible localization of the $\infty$-category of presheaves on a small $\infty$-category. In other words, there exists a small $\infty$-category $\C$ and an adjunction
\[
\begin{tikzcd}
L: \PSh (\C) \ar[r, shift left] &  \E\ar[l, shift left, hook] :i
\end{tikzcd}
\]
where $i$ is fully faithful and $L$ is an accessible functor (ie. commutes with $\kappa$-filtered colimits for some regular cardinal $\kappa$). In this situation, we say that $L:\PSh(\C) \to \E$ is an accessible localization functor.

\begin{definition}
    An $\infty$-category $\E$ is an \textit{
    $\infty$-topos} if there exists a small $\infty$-category $\C$ and an accessible left-exact localization functor $\PSh(\C) \longrightarrow \E$.
\end{definition}

An $\infty$-topos is in particular a presentable $\infty$-category and hence admits all small limits and colimits (\cite[Proposition~5.5.2.4]{HTT}). Higher toposes are the $\infty$-categorical analogues of \textit{Grothendieck toposes}. There is also a notion of \textit{elementary $\infty$-topos} that is developped in \cite{elementary}. The essential difference is that, as in the 1-categorical case, an elementary $\infty$-topos need not be presentable. In general, an elementary $\infty$-topos only has finite limits and colimits. Since the presentability assumption will be crucial in this work, we will consider only $\infty$-toposes in the sense of Lurie.

\begin{example}
    The following are examples of $\infty$-toposes: \begin{enumerate}
        \item The $\infty$-category $\Sp$ of spaces is an $\infty$-topos as it is the $\infty$-category of presheaves on the terminal category $*$.
        \item More generally, any presheaf $\infty$-category $\PSh (\C)$ is an $\infty$-topos.
        \item If $(\C, \CMcal T)$ is an $\infty$-site (ie. $\C$ is a small $\infty$-category and $\CMcal T$ is a Grothendieck topology on the homotopy category $h\C$), then $\mathrm{Sh}_{\CMcal T}(\C)$ is an $\infty$-topos. Contrary to the 1-categorical case, not every $\infty$-topos arises this way. In particular, from the data of a 1-category $\C$ and a Grothendieck topology $\CMcal T$ on $\C$, there is an $\infty$-topos $\mathrm{Sh}_\infty (\C, \CMcal T)$ of higher sheaves on $\C$, whose full subcategory of 0-truncated objects is the usual 1-topos $ \mathrm{Sh}(\C, \CMcal T)$ of sheaves on $(\C, \CMcal T)$ \cite[Theorem~6.4.1.5]{HTT}.
    \end{enumerate}
\end{example}

\begin{remark}
    Any $\infty$-topos has universal pushouts. In fact, any locally Cartesian closed $\infty$-category has universal pushouts. This follows from the fact that in such categories, the base change functors are left adjoints and hence preserves small colimits.
\end{remark}

\subsection{Group objects in $\infty$-categories.}

We denote by $\mathbf \Delta$ (a skeleton of) the category of non-empty finite linearly ordered sets and non-decreasing maps. The object $\{0< 1< \dots< n \}$ is written $[n]$. We view $\mathbf{\Delta} $ as an $\infty$-category through the nerve embedding $\Cat\hookrightarrow \Cat_\infty$. If $\C$ is an $\infty$-category, we write $s\C$ for the $\infty$-category $\mathrm{Fun}(\mathbf{\Delta}^\mathrm{op}, \C)$ of simplicial objects in $\C$. For a simplicial object $X_\bullet\in \s\C$, we denote by $X_n$ the value of $X_\bullet$ at $[n]$. 

Any bicomplete $\infty$-category is tensored and cotensored over $\Sp$. This (co)tensoring can be restricted to $\Set\simeq \tau_{\leq 0} \Sp \subseteq \Sp$ yielding formulas $$
X\otimes S = \coprod_{s\in S}X \:\:\:\:\:\:\:\:\:\:\:\: X^S=\prod_{s\in S}X
$$
for any set $S$ and object $X\in \C$.

\begin{definition}
    Let $X_\bullet\in \s\C$ be a simplicial object in $\C$. Given a simplicial set $K\in \sSet$, define an object $X(K)\in \C$ by the following end formula $$
    X(K)=\int_{[n]\in \DDelta} X_n^{K_n}.
    $$
    This assignment extends to a functor $\s\C \times \sSet^{\op}\longrightarrow \C$. 
\end{definition}
\begin{remark}
    Fixing the first variable $X\in \s\C$, the resulting functor $X(-):\sSet^\op \longrightarrow \C$ is the right Kan extension of $X$ along the Yoneda embedding $\DDelta \subseteq \sSet$.
\end{remark}
\begin{definition}\label{Segal}
    Let $X_\bullet \in \s\C$ be a simplicial object. \begin{enumerate}
        \item We say that $X_\bullet$ is a \textit{Segal object} if for each $n\geq 0$ and $0<i<n$, the inner horn inclusion $\Lambda^n_i \subseteq \Delta^n$ induces an equivalence $X(\Delta^n)\longweak X(\Lambda^n_i)$. 
        \item[] A \textit{monoid object} in $\C$ is a Segal object $X\in \s\C$ for which $X_0\simeq *$.
        \item We say that $X_\bullet$ is a \textit{grouplike Segal object} if for each $n\geq 0$ and $0\leq i \leq n$, the horn inclusion $\Lambda^n_i \subseteq \Delta^n$ induces an equivalence $X(\Delta^n)\longweak X(\Lambda^n_i)$. 
        \item[] A \textit{group object} in $\C$ is a grouplike Segal object $X\in \s\C$ for which $X_0\simeq *$.
    \end{enumerate}
    Denote by $\Grp(\C)\subseteq \Mon(\C)\subseteq \s\C$ the full subcategories on group objects and monoid objects respectively.
\end{definition}

\begin{remark}
    The Segal condition is more commonly formulated in terms of the so called \textit{Segal maps} $$
    X_n \longrightarrow X_1 \times_{X_0} \cdots \times_{X_0} X_1,
    $$
    which are induced by the various inclusions $[1] \to [n]$. These maps can be neatly expressed as being the images of the inclusions $$
    \Delta^1 \coprod_{\Delta^0} \cdots \coprod_{\Delta^0} \Delta^1 = \sk_1 \Delta^n \subseteq \Delta^n
    $$
    under the functor $X(-)$.
\end{remark}

    Monoids in $\C$ are $\EE_1$-algebra objects in the Cartesian monoidal category $\C^\times$. In \cite{HA} the category $\Mon(\C)$ is denoted $\Alg_{\EE_1}(\C)$. Similarly, group objects in $\C$ are the group-like $\EE_1$-algebra objects in $\C$, sometimes named $\EE_1$-groups. 

    When $\C$ is taken to be a $1$-category, monoid and group objects in $\C$ reduce to the usual notion of monoid and group objects in $\C$. Moreover, if $M_\bullet$ is a group object in $\C$, the set of homotopy classes $[X, M_1]$ comes equipped with a group structure which is natural in the object $X\in \C$.

\begin{remark}
The extended Segal conditions imply the Segal conditions, so that group objects in $\C$ are in particular monoid objects in $\C$. There are forgetful functors $\Grp(\C)\longrightarrow \Mon(\C) \longrightarrow \C_*$ where the rightmost arrow is given on objects by $M_\bullet \mapsto (* \overset{s_0}{\longrightarrow}M_1)$.
\end{remark}

\begin{remark}
For a monoid object $M_\bullet$ in $\C$, the object $M_1$ is the underlying object of $M_\bullet$. The multiplication on $M$ is the face map $d_1: M_1\times M_1 \longrightarrow M_1$ while the unit is given by the degeneracy $s_0: *\longrightarrow M_1$. These forgetful functors all admit left adjoints which are the corresponding free constructions. These left adjoints are informally given by 
$$
X \mapsto \coprod_{n\in \N} X^{\wedge n} \:\:\:\:\:\:\:\:\:\:\:\: M_\bullet \mapsto \Omega |M_\bullet|
$$
for the free monoid on a pointed object and the group completion of a monoid respectively.
\end{remark}

The following splitting result is stated without proof in \cite{EHP}.

\begin{proposition}\label{split}
Let $\C$ be a category with finite limits, and let $$F_\bullet \overset{i_\bullet}{\longrightarrow} E_\bullet \overset{p_\bullet}{\longrightarrow} B_\bullet$$ be a fiber sequence in $\Grp(\C)$. For any section $s:B_1 \longrightarrow E_1$ of the map $p_1$ in $\C_*$, the composite $$\begin{tikzcd}
F_1 \times B_1 \ar[r, "i_1\times s_1"] & E_1 \times E_1 \ar[r, "d_1"] & E_1
\end{tikzcd}
$$
is an equivalence.
\end{proposition}

\begin{proof} The sequence extends to a fiber sequence of pointed objects 
$$\begin{tikzcd}
\Omega E_1 \ar[r, "\Omega p_1"] & \Omega B_1 \ar[r, "\partial"] & F_1 \ar[r, "i_1"] & E_1 \ar[r, "p_1"] & B_1
\end{tikzcd}
$$
where the map $\Omega p_1 : \Omega E_1 \longrightarrow \Omega B_1$ has a section given by $\Omega s$. It follows that the map $\partial$ in the above sequence is nullhomotopic, since $\partial \simeq \partial \circ (\Omega p_1)\circ (\Omega s)\simeq *$, and hence the induced map on homotopy classes of maps is zero. Taking homotopy classes of maps gives a short exact sequence of honest groups $$
0 \longrightarrow [Z,F_1]_*\overset{(i_1)_*}{\longrightarrow} [Z,E_1]_* \overset{(p_1)_*}{\longrightarrow} [Z,B_1]_* \longrightarrow 0
$$
with a set-theoretical section. By a standard argument, the composite map 
\[
\begin{tikzcd}
{[Y,F_1]_*\times [Z, B_1]_*} \ar[rr, "(i_1)_*\times (s_1)_*"] && {[Z,E_1]_* \times [Z, E_1]_*} \ar[r, "(d_1)_*"] & {[Z,E_1]_*}
\end{tikzcd}
\]
is a bijection of sets. Since $(d_1)_*\circ ((i_1)_*\times s_*)\simeq (d_1\circ (i_1\times s))_*$ in the arrow category of $h\C_*$, it follows by the Yoneda lemma that the map $d_1\circ (i_1\times s)$ is an isomorphism in $h\C_*$. The result follows since maps that are sent to isomorphisms by the projection $\C_* \longrightarrow h\C_*$ onto the homotopy category are (by definition) equivalences in $\C_*$.
\end{proof}

\section{Homotopy theory in $\infty$-categories}

We write $\mathbf{\Delta}_a$ for the \textit{augmented simplex category}, defined as the category of all (possibly empty) finite linearly ordered sets and non-decreasing maps. There is an obvious inclusion $\mathbf{\Delta}\subseteq \mathbf{\Delta}_a$. 


\subsection{Smash product and half-smash product.}

In this section $\C$ denotes an $\infty$-category with finite limits and colimits.

\begin{definition}
\begin{enumerate}
    \item The \textit{suspension} of an object $ X\in \C$ is defined as the pushout 
    $$
    \begin{tikzcd}
        X \arrow[dr, phantom, "\scalebox{1}{$\ulcorner$}" , very near end, color=black]\ar[r]\ar[d] & * \ar[d] \\
        * \ar[r] & \Sigma X  .
    \end{tikzcd}
    $$
    \item[] The suspension operation defines a functor $\Sigma:  \C \longrightarrow \C$.
    \item The \textit{loop object} of an object $X\in \C_*$ is defined as the pullback 
    \[
    \begin{tikzcd}
        \Omega X \arrow[dr, phantom, "\scalebox{1}{$\lrcorner$}" , very near start, color=black]\ar[r]\ar[d] & * \ar[d] \\
        * \ar[r] & X  .
    \end{tikzcd}
    \]
    \item[] The looping operation defines a functor $\Omega:  \C_* \longrightarrow \C_*$.
    \item The \textit{half-smash product} of $X \in \CMcal {C}$ and $Y\in \CMcal {C_*}$ is defined as the pushout $$
    \begin{tikzcd}
        X \arrow[dr, phantom, "\scalebox{1}{$\ulcorner$}" , very near end, color=black]\ar[r, "{(1, y)}"]\ar[d] & X\times Y \ar[d] \\
        * \ar[r] & X \ltimes Y  .
    \end{tikzcd}
    $$
    \item[] The half-smash product defines a functor $\ltimes: \CMcal {C\times C_* \longrightarrow C_*}$.
    \item The \textit{smash product} of $X,Y \in \CMcal {C_*}$ is defined as the pushout 
    $$
    \begin{tikzcd}
        X\vee Y \arrow[dr, phantom, "\scalebox{1}{$\ulcorner$}" , very near end, color=black]\ar[r]\ar[d] & X\times Y \ar[d] \\
        * \ar[r] & X \wedge Y  .
    \end{tikzcd}
    $$
    \item[] This defines a functor $\wedge : \C_* \times \C_* \longrightarrow \C_*$.
    \item[] The \textit{join} of $X,Y\in C$ is defined by the pushout $$
    \begin{tikzcd}
        X\times Y \arrow[dr, phantom, "\scalebox{1}{$\ulcorner$}" , very near end, color=black]\ar[r]\ar[d] & X \ar[d] \\
        Y \ar[r] & X \star Y  .
    \end{tikzcd}
    $$ 
    \item[] The join defines a functor $\star:\C_* \times \C_* \longrightarrow \C_*$.
\end{enumerate}
\end{definition}
     
\begin{remark}
    When $\C$ is Cartesian closed, the half-smash product defines a left-tensoring of $\C$ over $\C_*$ in the following sense: for any objects $X\in \C_*$ and $Y \in \C$, the functors $X \ltimes -$ and $-\ltimes Y$ admit right adjoints given respectively by $\Map(X,-)$ and $\Map_*(Y,-)$, seen as functors to pointed objects by assigning the constant map as a basepoint for both $\Map(X,Z)$ and $\Map_*(X,Z)$.  
\end{remark}

\begin{proposition}\label{symmon}
    Let $\C$ be a Cartesian closed $\infty$-category. The smash product is part of a closed symmetric monoidal structure on $\C_*$. The unit is given by $S^0=*\sqcup *$. 
    
    In particular, the functor $\wedge : \C_* \times \C_* \longrightarrow \C_*$ preserves colimits in both variables.
\end{proposition}
\begin{proof}
    Under the assumption that $\C$ is Cartesian closed, the Cartesian product $\C \times \C \longrightarrow \C$ preserves colimits in both variables and \cite[Proposition~2.11]{gla} ensures the existence of a symmetric monoidal structure on $\C^\to=\Fun(\Delta^1, \C)$ given by Day convolution. As a result, $\C_*$ inherits a symmetric monoidal structure since $\cof:\C^\to \longrightarrow \C_*$ exhibits $\C_*$ as a reflective localization of $\C^\to$. It is then proved in \cite[Proposition~2.4.2]{abstractmotivic} that this induced monoidal structure is given on objects by the smash product $\wedge$.
\end{proof}

\begin{remark}
The smash product fits into the following pasting of pushout squares $$
\begin{tikzcd}
Y \arrow[dr, phantom, "\scalebox{1}{$\ulcorner$}" , very near end, color=black]\ar[r]\ar[d] & X \times Y \arrow[dr, phantom, "\scalebox{1}{$\ulcorner$}" , very near end, color=black] \ar[d] \ar[r] & X\ltimes Y \ar[d] \\
* \ar[r] & X \rtimes Y \ar[r] & X \wedge Y   
\end{tikzcd}
$$
where the right square is coCartesian. Indeed, note that in the following diagram $$\begin{tikzcd}
X \ar[d]  &  * \ar[l]\ar[d] \ar[r] &  Y \ar[d]  \\
X\times Y \ar[d] & X\times Y \ar[r]\ar[d, equals]\ar[l] &  X\times Y \ar[d] \\
X \ltimes Y & X\times Y \ar[l] \ar[r] & X \rtimes Y 
\end{tikzcd}
$$
all the columns are cofiber sequences. By \ref{fubini} the induced maps on colimits of the rows form a cofiber sequence $X\vee Y \longrightarrow X\times Y \longrightarrow X\wedge Y$, whence 
$$
X\wedge Y \simeq \colim(X\rtimes Y \leftarrow X\times Y \longrightarrow X\ltimes Y).
$$
\end{remark}

\begin{lemma}\label{join}
Let $\C$ be an $\infty$-category with finite products and pushouts, and let $X,Y \in \C_*$ be pointed objects of $\C$. There is a natural equivalence of pointed objects
$$
X \star Y \simeq \Sigma (X \wedge Y).
$$
\end{lemma}

\begin{proof}
Consider the following commutative diagram in $\C_*$
\[
\begin{tikzcd}
*  &  X \ar[l] \ar[r, equals] &  X   \\
* \ar[d, equals]\ar[u, equals] & X\vee Y \ar[r]\ar[d]\ar[l]\ar[u] &  X\times Y \ar[d]\ar[u] \\
* & Y \ar[l] \ar[r, equals] & Y  
\end{tikzcd}
\]
where all non-identity maps are projections. To compute its colimit one can first take colimits of the rows. We find the diagram $* \leftarrow X\wedge Y \longrightarrow *$ whose colimit is $\Sigma(X \vee Y)$. We could also take colimits of the columns first, which yields $* \leftarrow * \longrightarrow X \star Y$. The fact that $\mathrm{colim}(X \leftarrow X\vee Y \longrightarrow Y)\simeq *$ is explained by repeatedly applying the pasting law \ref{pasting} to  prove that in the following diagram$$
\begin{tikzcd}
* \arrow[dr, phantom, "\scalebox{1}{$\ulcorner$}" , very near end, color=black]\ar[r]\ar[d] & X \arrow[dr, phantom, "\scalebox{1}{$\ulcorner$}" , very near end, color=black] \ar[d] \ar[r] & * \ar[d] \\
Y \arrow[dr, phantom, "\scalebox{1}{$\ulcorner$}" , very near end, color=black] \ar[r]\ar[d] & X\vee Y \arrow[dr, phantom, "\scalebox{1}{$\ulcorner$}" , very near end, color=black] \ar[r]\ar[d] & Y \ar[d] \\  
* \ar[r] & X \ar[r] & *
\end{tikzcd}
$$
all the squares are coCartesian.
\end{proof}

\begin{proposition}
Let $\C$ be an $\infty$-category with finite products and universal pushouts. For any pointed object $X\in \C_*$, there is a natural equivalence of pointed objects
$$
\Sigma X \simeq X \wedge S^1.
$$
\end{proposition}
\begin{proof} Consider the following commutative diagram of pointed objects in $\C$ 
$$
\begin{tikzcd}
*  &  X \ar[l] \ar[r, equals] &  X   \\
* \ar[d, equals]\ar[u, equals] & X_+ \ar[r]\ar[d]\ar[l]\ar[u] &  X\sqcup X \ar[d]\ar[u] \\
* & X \ar[l] \ar[r, equals] & X  
\end{tikzcd}
$$
where $X_+=X\sqcup *$ is $X$ with a disjoint basepoint, the two maps $X\sqcup X \longrightarrow X$ are both the codiagonal $\nabla_X = (1,1)$ on $X$ and the map $X_+ \longrightarrow X$ is induced by $\id_X$ and the base point $x:*\longrightarrow X$ by the universal property of coproducts in $\C$. Taking for $S^0=*\sqcup *$ any choice of basepoint, there are equivalences of pointed objects $X\times S^0\simeq X\sqcup X$ and $X\vee S^0 \simeq X_+$, where $X\sqcup X$ is given the basepoint corresponding to the one chosen for $S^0$. Note that $S^1=\Sigma S^0\simeq \colim(* \leftarrow S^0 \to *)$. 

Note also that $X\simeq \colim(* \leftarrow X_+ \to X\sqcup X)$. This follows from \ref{pasting} applied to 
$$\begin{tikzcd}
* \arrow[dr, phantom, "\scalebox{1}{$\ulcorner$}" , very near end, color=black]\ar[r]\ar[d] & X_+ \arrow[dr, phantom, "\scalebox{1}{$\ulcorner$}" , very near end, color=black] \ar[d] \ar[r] & * \ar[d] \\
X \ar[r] & X\sqcup X \ar[r] & X  
\end{tikzcd}
$$
where the bottom composite is the identity $\id_X$.
Taking colimits of the rows we obtain the diagram $* \leftarrow X \to *$ because $X\times -$ commutes with (unpointed) colimits. Taking colimits of the columns we find $* \leftarrow X\vee S^1 \to  X \times S^1$. The result follows from \ref{fubini}.
\end{proof}

\begin{corollary}
Let $\C$ be an $\infty$-category with finite products and universal pushouts. For any pointed objects $X,Y \in \C_*$, there are natural equivalences of pointed objects
$$
\Sigma X \wedge Y \simeq \Sigma (X\wedge Y) \simeq X \wedge \Sigma Y .
$$
\end{corollary}
\begin{proof} This follows from the associativity and commutativity of the smash product \Cref{symmon}: 
$$
(\Sigma X )\wedge Y \simeq   S^1\wedge X \wedge Y \simeq \Sigma (X \wedge Y).
$$
\end{proof}

\begin{proposition}\label{cof}
    Let $\C$ be an $\infty$-category with pushouts. If $f:X\longrightarrow Y$ is a nullhomotopic map in $\C_*$, then there is an equivalence of pointed objects 
    $$\cof(f)\simeq \Sigma X \vee Y$$
    which is natural in $f\in \C^\to$.
\end{proposition}
\begin{proof} Being nullhomotopic $f$ factors as $X\longrightarrow * \longrightarrow Y$. Now apply the Pasting law \ref{pasting} to the following diagram$$
\begin{tikzcd}
X \arrow[dr, phantom, "\scalebox{1}{$\ulcorner$}" , very near end, color=black]\ar[r]\ar[d] & * \arrow[dr, phantom, "\scalebox{1}{$\ulcorner$}" , very near end, color=black] \ar[d] \ar[r] & Y \ar[d] \\
* \ar[r] & \Sigma X \ar[r] & \Sigma X \vee Y.   
\end{tikzcd}
$$
\end{proof}

\begin{proposition}\label{presentability}
    Suppose given an adjoint pair of functors 
    \[
    \begin{tikzcd}
    L: \Y \ar[r, shift left] &  \X \ar[l, shift left, hook] :j
    \end{tikzcd}
    \]
    in which the right adjoint $j$ is fully faithful.
    \begin{enumerate}
        \item There is an induced adjoint pair between categories of pointed objects
        $$
        \begin{tikzcd}
        L_*: \Y_* \ar[r, shift left] &  \X_* \ar[l, shift left, hook] :j_*
        \end{tikzcd}
        $$
        in which $j_*$ is fully faithful.
        \item If $L$ is left exact, then $L_*$ is also left-exact and commutes with smash products:  
        $$
        L_*(X\wedge Y)\simeq L_*X \wedge L_*Y
        $$ for any pointed objects $X,Y \in \Y_*$.
    \end{enumerate}
\end{proposition}

\begin{proof}
(1) Since $j$ is fully faithful we have
    $$
    \Map_\X(X, L*) \simeq \Map_\X (LjX, L*) \simeq \Map_\Y(jX, *) \simeq *
    $$
    for any $X \in \X$, hence $L*$ is a terminal object in $\X$. This shows that $L$ preserves terminal objects, and hence descends to a functor $L_*:\Y_* \longrightarrow \X_*$ between coslice categories. The same holds for $j$ since it is a right adjoint. Now for pointed objects $X\in \X_*$ and $Y\in \Y_*$, the natural isomorphism $$
    \Map_{\Y_*}(j_*X, Y) \simeq \Map_{\X_*} (X, L_*Y)
    $$
    results from the following natural equivalence of cospans of spaces 
    \[
    \begin{tikzcd}
         \Map_\Y(jX, Y) \ar[r] \ar[d, "{\sim}" sloped] & \Map_\Y(*, Y) \ar[d, "{\sim}" sloped] & \{y\} \ar[l] \ar[d, "{\sim}" sloped] \\
         \Map_\X(X, LY)\ar[r] & \Map_\X(*, LY)  & \{x\} \ar[l]  
    \end{tikzcd}
    \]
    where the colimits of the top and bottom spans are respectively $\Map_{\Y_*}(j_*X, Y)$ and $\Map_{\X_*} (X, L_*Y)$. The middle vertical map is the composite $$
    \Map_\Y(*, Y) \longweak \Map_\Y(j*, Y) \longweak \Map_\X(*, LY)
    $$
    in which the left arrow is induced by composition with the equivalence $j* \longweak *$, and the right equivalence results from the adjunction $L \dashv j$. The fact that $j_*$ is fully faithful follows from a similar argument. 
    \\ 
    (2) Left-exactness of $L_*$ is immediate since limits in coslice categories are computed at the level of underlying objects, and $L$ preserves finite limits by assumption. Moreover, since $L_*$ commutes with small colimits and finite limits, the square 
    $$\begin{tikzcd}
    L_*X\vee L_*Y \arrow[dr, phantom, "\scalebox{1}{$\ulcorner$}" , very near end, color=black] \ar[r] \ar[d] & L_*X \times L_*Y \ar[d] \\
    * \ar[r] & L_*(X\wedge Y)
    \end{tikzcd}
    $$
    is a pushout in $\X$, and the results follows.
\end{proof}

\subsection{Connectivity structure}
In this paragraph, we record the definition of $n$-connected maps in $\infty$-toposes, following \cite[Section~6.5]{HTT}. Though the definition we use makes sense in an arbitrary $\infty$-category with finite limits, the connectivity estimate on smash productd \Cref{connprop} requires stability properties of $n$-connected maps which we only know to hold in an $\infty$-topos.

\begin{definition}\label{cech}
    Let $\C$ be an $\infty$-category with finite limits and let $f:X\longrightarrow Y$ be a map in $\C$.
    \begin{enumerate}
        \item 
    The \emph{\v Cech nerve} of $f$ is the augmented simplicial object $\check{C} (f): \mathbf{\Delta}_a \longrightarrow \C$ that is the right Kan extension of $f:[1]\longrightarrow \C$ along the inclusion $[1] \hookrightarrow \mathbf{\Delta}_a$.
    The \v Cech nerve of $f$ can be represented by $$
    \begin{tikzcd}
\cdots \ar[r, shift right=3]\ar[r, shift right=1]\ar[r, shift left=1]\ar[r, shift left=3]
&  X\times_Y X \times_Y X 
\ar[r]\ar[r, shift right=2]\ar[r, shift left=2]
& X\times_Y X 
\ar[r, shift right]\ar[r, shift left] & X \ar[r, "f"] & Y
\end{tikzcd}
    $$
    where face maps are projections and degeneracies are formed using the diagonal $\Delta f$.
    \item $f$ is an \textit{effective epimorphism} if $f$ exhibits $Y$ as a colimit of the restriction $$\begin{tikzcd}
    \mathbf{\Delta} \subseteq \mathbf{\Delta}_a \ar[r, "{\check{C} (f)}"] & \C.
    \end{tikzcd}
    $$
    \end{enumerate}
\end{definition}

\begin{definition}\label{conndef}
    Let $\C$ be an $\infty$-category and $f:X\longrightarrow Y$ be a map in $\C$.
    \\ 
    We define the property of $f$ being an $n$\textit{-connected} map, written $\conn(f)\geq n$, by induction on $n\geq -1$ as follows: \begin{enumerate}
        \item $f$ is $(-1)$-connected if it is an effective epimorphism.
        \item $f$ if $n$-connected if it is $(-1)$-connected and the diagonal $\Delta f : X \longrightarrow X\times_Y X$ is $(n-1)$-connected.
    \end{enumerate}
    
    An object $X\in \C$ is called $n$-connected if the unique map $X\longrightarrow *$ is $n$-connected. If $X$ is an $n$-connected object, we write $\conn(X)\geq n$. We say that an object is \textit{connected} when it is 0-connected.
\end{definition}

\begin{remark}
    If $f:X \longrightarrow Y$ is an $n$-connected map in $\C_*$, then it is $n$-connected in $\C$. The converse is not true in general, as can be noticed when $\C=\Sp$.

    In a presheaf $\infty$-category $\PSh (\C)$, an object $X$ is $n$-connected if and only if the spaces $X(C)$ are $n$-connected for any $C\in \C$. That is, if is a presheaf of $n$-truncated (resp. $n$-connected) spaces. The same holds for $\infty$-categories of sheaves on a site.
\end{remark}

\begin{remark}
    When $\C=\Sp$, an object $X$ is $n$-connected if and only if it is $n$-connected in the usual sense, ie. if its homotopy groups $\pi_i(X)$ vanish for all $i<n$ for any choice of basepoint. There is however a shift in the convention for maps: a map $f$ in $\Sp$ is $n$-connected according to \Cref{conndef} exactly when $\pi_i(f)$ is an isomorphism for $i\leq n$ and an epimorphism for $i=n+1$. These maps are called $(n+1)$-connected in the traditional use of the term in homotopy theory. The terminology we use is in accordance with \cite{ABJF20} and \cite{EHP}, but differs from \cite{HTT} where an $n$-connected map in our sense is called $(n+1)$-connective.
\end{remark}

In the rest of this section, we let $\E$ be an $\infty$-topos and $n\geq -1$ an integer. 

\begin{proposition}\label{connpullback}\cite[Section~6.5.1]{HTT}
    \begin{enumerate}
        \item  The class of $n$-connected maps is stable under base change and cobase change along any morphism.
        \item Suppose $X \overset{f}{\longrightarrow} Y \overset{g}{\longrightarrow} Z$ are composable maps in $\E$ where $f$ is $n$-connected. Then $g$ is $n$-connected if and only if $g\circ f$ is $n$-connected.
        \item Suppose $f:X\to Y$ is a map that admits a section $s$. Then $f$ is $(n+1)$-connected if and only if $s$ is $n$-connected.
    \end{enumerate}  
\end{proposition}

\begin{proposition}\label{connlemma}
Let $X,Y \in \E_*$ be pointed objects of $\E$. The following statements are equivalent:
    \begin{enumerate}
        \item $X$ is a $k$-connected object of $\E$.
        \item The map $*\longrightarrow X$ is $(k-1)$-connected.
        \item The map $(\id, *):Y \longrightarrow Y \times X$ is $(k-1)$-connected.
    \end{enumerate}
\end{proposition}
\begin{proof}
    Statements (1) and (2) are equivalent by \Cref{connpullback} since $* \longrightarrow X$ is a section of the map $X\to *$. Regarding (3), consider for each $Y\in \E_*$ the diagram $$
    \begin{tikzcd}
        X \arrow[dr, phantom, "\scalebox{1}{$\lrcorner$}" , very near start, color=black]\ar[r]\ar[d] & Y\times X \arrow[dr, phantom, "\scalebox{1}{$\lrcorner$}" , very near start, color=black] \ar[d] \ar[r] & X \ar[d] \\
        * \ar[r] & Y \ar[r] & *  
    \end{tikzcd}
    $$
    in which both squares are pullbacks. \Cref{connpullback} ensure that (1) and (3) are equivalent since $(\id, *)$ is a section of the projection $Y\times X \longrightarrow X$.
\end{proof}

\begin{proposition}\label{connprop}
    Let $\E$ be an $\infty$-topos and $X,Y\in \E_*$ be pointed objects. Assume that $X$ is $k$-connected and $Y$ is $\ell$-connected. Then :\begin{enumerate}
    \item $\Sigma X$ is $(k+1)$-connected
    \item $\Omega X$ is $(k-1)$-connected
    \item $X\wedge Y$ is $(k+\ell +1)$-connected
    \item $X^{\wedge n}$ is $(n(k+1)-1)$-connected for any $n\geq0$.
    \end{enumerate}
\end{proposition}

\begin{proof} The proof of (1) and (2) follows from the fact that $n$-connected morphisms are stable under (co)base change, and that a pointed object $X\in \E_*$ is $k$-connected if and only if $*\longrightarrow X$ is $(k+1)$-connected in $\E$. Notably, statements (3) and (4) use the fact that since $\E$ is an $\infty$-topos, $n$-connected maps form the left class of a factorization system. A complete proof is in \cite[Proposition~4.12]{EHP}.
\end{proof}

\subsection{Hypercompleteness}

In this section $\E$ denotes an $\infty$-topos.

\begin{definition}
    We say that a map $f:X\longrightarrow Y$ in $\E$ is \emph{$\infty$-connected} if it is $n$-connected for every $n\geq -2$.
\end{definition}

\begin{proposition}\label{inftyconnected}
    Suppose given a tower of maps in $\E$ of the form $$
    \begin{tikzcd}
        X_1 \ar[d]\ar[r] & X_2 \ar[d]\ar[r] & X_3 \ar[d]\ar[r] & \cdots  \\
        Y_1 \ar[r] & Y_2 \ar[r] & Y_3 \ar[r] & \cdots 
    \end{tikzcd}$$
    and a sequence of integers $(k_n)$ such that $k_n \to \infty$ and the maps $X_n \to X_{n+1}$ and $X_n \to Y_n$ are $k_n$-connected for each $n$. Then the induced map on colimits $X_\infty \longrightarrow Y_\infty$ is $\infty$-connected.  
\end{proposition}
\begin{proof}
    Taking an appropriate subdiagram if necessary, we can assume that $(k_n)$ is a non-decreasing sequence. Under this assumption, the map $Y_n \to Y_{n+1}$ is also $k_n$-connected by \Cref{connpullback}. For each $n$, consider the square $$
    \begin{tikzcd}
        X_n \ar[r] \ar[d] & X_{\infty} \ar[d] \\
        Y_n \ar[r] & Y_{\infty}.
    \end{tikzcd}
    $$
    The maps $Y_n \to Y_\infty$ and $X_n \to X_\infty$ are $k_n$-connected as composites of $k_n$-connected maps. Hence the composite $X_n \to Y_n \to Y_\infty$ is $k_n$-connected, and \Cref{connpullback} ensures that $X_\infty \to Y_\infty$ is $k_n$-connected. This concludes the proof.
\end{proof}

\begin{definition}
    An $\infty$-topos $\E$ is called \textit{hypercomplete} if any $\infty$-connected map is an equivalence.
\end{definition}

\begin{example}
We list some basic examples of hypercomplete $\infty$-toposes:
\begin{enumerate}
    \item $\Sp$ is a hypercomplete $\infty$-topos. This is the content of \textit{Whitehead's theorem}. See for example \cite[Chapter~V, Theorem~3.5]{bookwhitehead} for a textbook account.
    \item Any presheaf $\infty$-topos $\PSh (\C)$ is also hypercomplete. However, sheaf $\infty$-toposes are generally \textit{not} hypercomplete (even on a 1-categorical site). See for example \cite[Section~11.3]{rezktopos} where hypercompleteness is called $t$-completeness.
    \item Any $\infty$-topos that is (equivalent to) an $n$-category for some $n<\infty$ is hypercomplete. This follows from \cite[Lemma~6.5.2.9]{HTT}.
\end{enumerate}   
\end{example}

\section{Free Lie algebras and Hall words}
\label{hallsection}

In this section we introduce \textit{Hall words}, and include a structure result which allows to compute them. These results are classical and we include them only for reference. What we call a set of Hall words on symbols $x_1, \cdots, x_n$ would traditionally be referred to as a \textit{Hall basis} on these symbols. This is (one choice of) a basis for the \textit{free Lie ring} (a Lie algebra over $\mathbb Z$) generated by the symbols~$x_i$.

For the rest of this section, we fix a totally-ordered finite set of symbols $L=\{x_1< \cdots< x_n\}$ and denote $\Lie(L)=\Lie (x_1, \cdots , x_n)$ the free Lie ring generated by the $x_i$. As proved for example in \cite[Proposition~8.5.1]{neisendorfer_2010}, the underlying abelian group of $\Lie (L)$ is free. The purpose of a Hall basis is to provide an explicit basis for this free $\mathbb Z$-module.

The construction of a Hall basis is inductive. At each step $n$ in the construction we build a countable list $L_n \subseteq \Lie (L)$ with a total order on it. As usual we write $\ad(x)(y)=[x,y]$ for the adjoint representation of a Lie algebra. Hence for $i\geq 0$, $\ad(x)^i(y)$ is the iterated commutator $[x, [x, \cdots, [x,y]\cdots]]$ where the symbol $x$ appears $i$ times. Given an iterated commutator $w$ involving elements of $L$, its wordlength $\ell(w)$ is the total number of elements of $L$ appearing in the expression of $w$, counted with multiplicity. For instance, $\ell([x,[x,y]])=3$.

\begin{definition}\label{hall_dfn}
    Let $L$ be a finite totally-ordered set. Define inductively a sequence $L_n\subseteq \Lie (L)$ of well-ordered sets as follows : 
    \begin{enumerate}
    \item Set $L_0=L$ with its given total order.
    \item If $L_{r-1}$ has been constructed and given a well order, define $x_{r}=\mathrm{min}\: L_{r-1}$ and set $$
    L_{r}=\{\ad(x_{r})^i(x) \: | \: i\geq 0, x\in L_{r-1}, x\ne x_{r}\}.
    $$
    Then choose \textit{any} total order $<$ on $L_{r}$ that is compatible with word length. That is, we require that $x<y$ whenever $\ell(x)<\ell(y)$ for elements $x,y \in L_{r}$. This implies that $<$ is a well order.
    \end{enumerate}
    The resulting sequence $B=(x_r)_{r\geq 1}$ is called a \textit{Hall basis} on the set $L$.
\end{definition}

\begin{remark}
Note that at each step in the construction of a Hall basis on $L$, we are free to choose any total order compatible with word length. Different total orders will result in different Hall bases. Hence a Hall basis is not uniquely determined by $L$. Note that, taking $i=0$ in the definition of $L_{r}$, we see that $L_{r-1}\setminus \{x_{r}\} \subseteq L_{r}$. A convenient convention is to take the order of $L_{r}$ to extend the one already defined on $L_r$. With this convention, denoting $L=\{x_1, \cdots, x_n\}$ we see that the first $n$ elements of a Hall basis are the generators $x_1 < \cdots < x_n$. 
\end{remark}

We include the following two results for the interest of the reader. First, as advertised above a Hall basis on $L$ is a $\mathbb Z$-basis for $\Lie(L)$. This is proved in \cite{Hall1950ABF}, where elements of a Hall basis are called standard monomials. 

\begin{theorem}{\rm \cite[Theorem 3.1]{Hall1950ABF}}
    Any Hall basis on $L$ provides a basis for the free abelian group underlying $\Lie(L)$. Denoting $B$ such a Hall basis, there is an isomorphism of abelian groups $$
    \Lie (L) \cong \mathbb Z [B].
    $$
\end{theorem}


The following gives an efficient algorithm for exhibiting words of length $\leq \ell$ in a Hall basis. 

\begin{proposition}
    Define an increasing sequence of finite totally-ordered sets $(B_\ell)_{\ell\geq 1}$ of elements of $\Lie (L)=\Lie(x_1, \cdots, x_n)$, together with a \textit{rank function} $r:B_\ell \longrightarrow \mathbb N$ inductively as follows : \begin{enumerate}
        \item Set $B_1=L=\{x_1< \cdots < x_n \}$ and $r(x_i)=0$ for $i=1, \cdots , n$.
        \item Suppose $B_{\ell-1}$ has been constructed and given a total order together with a rank function $r:B_{\ell-1} \to \mathbb N$. 
        \item[] Define $$
        B_{ \ell}=B_{\ell-1} \cup \big\{ \: [x_i,x_j] \: |\: r(x_j)\leq i < j \: , \: \ell(x_i)+ \ell(x_j)=\ell \: \big\}
        $$
        and chose any total order on $B_\ell$ extending the one on $B_{\ell-1}$.
        \item[] Extend the rank function to $B_{\ell}$ by setting $r([x_i, x_j])=i$.
    \end{enumerate}
    Then $B= \bigcup_{\ell \geq 1} B_\ell$ is a Hall basis on $L$.
\end{proposition}
\begin{proof}
    The set $B$ is countable and inherits a total order since the order on each $B_\ell$ is compatible with the inclusion $B_\ell \subseteq B_{\ell+1}$. Write $B=\big\{ x_1, x_2, \cdots \big\}$ where $x_i < x_j$ for $i<j$. An easy induction on the length of $x_i$ shows that $r(x_i)<i$ for every $i$. Now set 
    $$ 
    L_r=\big\{x\in B \: | \: r(x)\leq r < i \big\}
    $$ 
    for each $r\geq 0$. It is then clear that $x_{r+1}=\min (L_r)$ and we recover \Cref{hall_dfn} above by \cite[Remark~6.1]{bookwhitehead} which shows that 
    $$\begin{aligned}
    L_r &= \big\{ \ad(x_r)^i(x_j) \: | \: i\geq 0 , r(x_j) < r < j  \big\} \\
    &= \big\{ \ad(x_r)^i(x) \: | \: i\geq 0 , x_r \ne x \in L_{r-1}  \big\}
    \end{aligned}
    $$
    since $r(x_j) < j$ for every $j\geq 1$.
\end{proof}

\begin{example}
    To illustrate the description given above, let us list all words of length $\leq 5$ in a Hall basis on two generators denoted $x,y$.$$
    \begin{tabular}{|l|c|r|}
    \hline
        $\ell=1$ & $w_1=x$ & $r=0$ \\
                 & $w_2=y$ & $r=0$ \\
    \hline
        $\ell=2$ & $w_3=[x,y]$ & $r=1$ \\
    \hline
        $\ell=3$ & $w_4=[x,[x,y]]$ & $r=1$ \\
                 & $w_5=[y,[x,y]]$ & $r=2$ \\
    \hline
        $\ell=4$ & $w_6=[x,[x, [x,y]]]$ & $r=1$ \\
                 & $w_7=[y,[x, [x,y]]]$ & $r=2$ \\
                 & $w_8=[y,[y, [x,y]]]$ & $r=2$ \\
    \hline
        $\ell=5$ & $w_9=[x,[x, [x,[x,y]]]]$ & $r=1$ \\
                 & $w_{10}=[y,[x, [x,[x,y]]]]$ & $r=2$ \\
                 & $w_{11}=[y,[y, [x,[x,y]]]]$ & $r=2$ \\
                 & $w_{12}=[y,[y, [y,[x,y]]]]$ & $r=2$ \\
                 & $w_{13}=[[x,y],[x, [x,y]]]$ & $r=3$ \\
                 & $w_{14}=[[x,y],[y, [x,y]]]$ & $r=3$ \\
    \hline
    \end{tabular}$$
\end{example}

\begin{notation}\label{words}
    We introduce the following useful notation: given a word $w$ in the free Lie ring $\Lie(x_1, \cdots, x_n)$ and given pointed objects $X_1, \cdots, X_n \in \C_*$ in an $\infty$-category $\C$ with finite products, we denote $w(X_1, \cdots, X_n)$ the object of $\C_*$ obtained by smashing together the $X_i$ in the order indicated by $w$. 
    
    For example, $[x,y](X,Y)=X\wedge Y$ and if $w=[[x_1, [x_3, x_4]], [x_1, [x_1, [x_2, x_3]]$, then 
    $$\begin{aligned}
    w(X,Y,Z,T) &=(X\wedge (Z \wedge T)) \wedge (X \wedge (X\wedge (Y\wedge Z))) \\
    &\simeq X^{\wedge 3}\wedge Y \wedge Z^{\wedge 2} \wedge T.
    \end{aligned}$$
    Since the smash product is commutative and associative (up to natural equivalence), the order in which we smash is of course irrelevant, but this notation helps keep track of the different factors that appear in Theorem \ref{HM}. 
\end{notation}

\section{Hilton-Milnor's theorem}
Before turning to the proof of Hilton-Milnor's theorem, we establish several splitting results in the maximal generality of $\infty$-categories with suitable limits and colimits. Though we state our main theorem \Cref{HM} in an $\infty$-topos, some of the following splittings hold true in a greater generality and may be of independant interest.

\subsection{Fundamental splitting results}

The following analogue of James' splitting is proved in \cite{EHP}. Though the comparison map exists in any $\infty$-category with universal pushouts, the proof that it is an equivalence requires a presentability assumption as well as some connectivity estimates known to hold only in an $\infty$-topos.

\begin{proposition}\label{james}\cite[Proposition~4.18]{EHP}
    Let $\CMcal{E}$ be an $\infty$-topos, and $X\in \E_*$ a pointed connected object. There is a natural equivalence $$
    \Sigma \Omega \Sigma X \simeq \bigvee_{i=1}^\infty \Sigma X^{\wedge i}.
    $$
\end{proposition}

\begin{proposition}\label{half1}
Let $\C$ be an $\infty$-category with universal pushouts and let $X,Y \in \C_*$ be pointed objects in $\C$. There is a natural equivalence $$
\Omega (X \vee Y) \simeq \Omega Y \times \Omega(X \rtimes \Omega Y).
$$
\end{proposition}
\begin{proof} The natural projection map $X\vee Y \longrightarrow Y$ fits into a fiber sequence $$
X \rtimes \Omega Y \longrightarrow X \vee Y \longrightarrow Y.
$$
Indeed, consider the following diagram $$
\begin{tikzcd}
* \ar[d]  & \Omega Y \ar[l]\ar[d] \ar[r] & X \times \Omega Y \ar[d] \\
Y \ar[d] & * \ar[r]\ar[d]\ar[l] & X\ar[d, "*"] \\
Y  & Y \ar[l, equals] \ar[r, equals] & Y 
\end{tikzcd}
$$
where the columns are fiber sequences. Universality of pushouts (or \Cref{descent}) implies that, taking colimits of the rows, we get a fiber sequence $X \ltimes \Omega Y \longrightarrow X \vee Y \longrightarrow Y$.
This sequence admits a section given by the canonical inclusion map $ Y \longrightarrow X \vee Y$. Hence $$
\Omega (X \rtimes \Omega Y) \longrightarrow \Omega (X \vee Y) \longrightarrow \Omega Y
$$
is a split fiber sequence of group objects in $\C$. The result follows from \ref{split}. 
\end{proof}

\begin{proposition}\label{half2}
Let $\C$ be an $\infty$-category with products and universal pushouts and $X,Y \in \C_*$ be pointed objects in $\C$. There is a natural equivalence $$
\Sigma X \rtimes Y \simeq \Sigma (X \vee (X\wedge Y)).
$$
\end{proposition}
\begin{proof} Consider the commuting diagram : $$
\begin{tikzcd}
*  & X \ar[l] \ar[r, equals] & X  \\
Y \ar[d, equals]\ar[u] & X\times Y \ar[r, equals]\ar[d, equals]\ar[l]\ar[u] & X \times Y \ar[d]\ar[u] \\
Y  & X \times Y \ar[l] \ar[r] & Y 
\end{tikzcd}
$$
where all non-identity maps are projections. Taking colimits of the rows, we obtain $
* \longleftarrow Y \longrightarrow  \Sigma X \times Y
$
since $Y\times -$ preserves colimits, hence $$
\mathrm{colim}(Y \leftarrow X \times Y \longrightarrow Y ) \simeq Y \times \mathrm{colim} (* \leftarrow X \longrightarrow *) \simeq Y \times \Sigma X .
$$
Taking colimits of the columns, we get $* \longleftarrow X \longrightarrow \Sigma X\wedge Y$ by \ref{join}. Since $X\longrightarrow \Sigma X\wedge Y$ is nullhomotopic, its cofiber is canonically equivalent to $\Sigma X \vee (X * Y)\simeq \Sigma (X\vee (X\wedge Y))$ by \ref{cof}. Applying \ref{fubini}, we obtain the result.
\end{proof}

\begin{remark}
The above formula can also be obtained by noticing that in the following diagram 
$$
\begin{tikzcd}
X\times Y \arrow[dr, phantom, "\scalebox{1}{$\ulcorner$}" , very near end, color=black] \ar[r]\ar[d] & X \ar[d] \ar[r] & \Sigma X \rtimes Y \ar[d, "\sim" sloped] \\
 Y \ar[r] &  \Sigma  X \wedge Y  \ar[r] & \Sigma ( X \vee (X \wedge Y))
\end{tikzcd}
$$ 
both horizontal composites are cofiber sequences, and since the left square is a pushout it induces an equivalence on cofibers.
\end{remark}

The following statement is an immediate consequence of \ref{half1} and \ref{half2}. It is a variant of the fundamental Hilton-Milnor splitting \cite[Theorem~3.1]{EHP}.

\begin{corollary}\label{split1}
Let $\C$ be an $\infty$-category with products and universal pushouts and $X,Y \in \C_*$ be pointed objects in $\C$. There is a natural equivalence of pointed objects
$$
\Omega \Sigma (X\vee Y) \simeq \Omega \Sigma X \times \Omega \Sigma \big( Y \vee (Y \wedge \Omega \Sigma X) \big).
$$
\end{corollary}

We can now state the fundamental splitting formula, whose iteration yields Hilton-Milnor's theorem. Note that the connectivity asssumption on $X$ and $Y$ is essential to apply James' splitting $\ref{james}$.

\begin{proposition}\label{fundamental}
Let $\E$ be an $\infty$-topos, and $X, Y\in \E_*$ be pointed connected objects. Then there is a natural equivalence of pointed objects :$$
\Omega \Sigma (X \vee Y ) \simeq \Omega \Sigma X \times \Omega \Sigma \big( \bigvee_{i=0}^\infty Y \wedge X^{\wedge i} \big).
$$
\end{proposition}

\textit{Proof:} First apply Corollary \ref{split1} above, then use \ref{james} to split $\Sigma \Omega \Sigma X$ :
$$
\begin{aligned}
\Omega \Sigma (X \vee Y ) & \simeq \Omega \Sigma X \times \Omega (\Sigma Y \vee  (Y \wedge \Sigma \Omega \Sigma  X)) \\
& \simeq  \Omega \Sigma X \times \Omega (\Sigma Y \vee  (Y \wedge \bigvee_{i\geq1} \Sigma X^{\wedge i}))  \\
& \simeq \Omega \Sigma X \times \Omega \Sigma \big( \bigvee_{i\geq 0} Y \wedge X^{\wedge i} \big). \:\:\:\: \square
\end{aligned} 
$$

\begin{remark}
    \Cref{split1} bears strong resemblance to the `fundamental Hilton-Milnor splitting' \cite[Theorem~3.1]{EHP}, which takes the form $$
    \Omega \Sigma (X\vee Y)\simeq \Omega \Sigma X \times \Omega \Sigma Y \times \Omega \Sigma(\Omega \Sigma X \wedge \Omega \Sigma Y).
    $$
    As opposed to ours, this formula is symmetric in $X$ and $Y$, as well as the corresponding `general Hilton-Milnor splitting' \cite[Proposition~4.21]{EHP}. Though the proofs of both statements are very similar to ours, the asymmetric formulas are better suited to our purposes, given the iterative definition of Hall words \Cref{hall_dfn}. 
\end{remark}


\subsection{Main theorem}

Now that we have introduced all the necessary preliminaries, we are in a position to formulate and prove Hilton-Milnor's theorem in an $\infty$-topos. For the sake of completeness, we define the notion of weak product.

\begin{definition}\label{weakproduct}
    Let $\C$ be an $\infty$-category which has finite products and sequential colimits. Given a sequence $(X_n)_{n\geq 0}$ of pointed objects in $\C$, their \textit{weak product}
    $$
    \widetilde\prod_{i\geq 0} X_i
    $$
    is defined as the colimit of the sequential diagram of finite products 
    $$
    \begin{tikzcd}
        X_0 \ar[r, "{(\id, *)}"] & X_0 \times X_1 \ar[r] & \cdots \ar[r] & \prod_{i<n} X_i \ar[r, "{(\id, *)}"] & \prod_{i\leq n} X_i \ar[r] & \cdots
    \end{tikzcd}
    $$    
\end{definition}

We now come to the main theorem of this note.

\begin{theorem}\label{HM}
    Let $X_1,\cdots , X_n \in \E_*$ be pointed connected objects in an $\infty$-topos $\E$ and let $B(n)$ be a Hall basis on $n$ symbols $x_1, \cdots , x_n$. 
    
    Then there is a canonical equivalence of pointed objects $$
    \Omega \Sigma( X_1 \vee \cdots \vee  X_n) \simeq \widetilde\prod_{w\in B(n)} \Omega \Sigma w(X_1, \cdots ,X_n). 
    $$
\end{theorem}
\begin{proof} 

We divide the proof into the following steps: after setting up some notation, we construct a map that will realize the desired equivalence. We then show that this map is $\infty$-connected. Finally we show that the map is an equivalence.

\begin{step} We will be using the following notation throughout the course of the proof.
    \begin{enumerate}
        \item Given our choice of Hall basis $B=B(n)$, we denote by $L_r$ the set of elements of $B$ of rank $r$. We fix an order on $L_r$ compatible with wordlength, so that $x_{r+1}=\mmin(L_r)$.
        \item Consistently with \Cref{words}, we write $X_i=x_i(X_1, \cdots, X_n)$ for the object of $\E_*$ obtained by smashing $X_1, \cdots, X_n$ in the order corresponding to the word $x_i$.
        \item For each $r\geq 0$, we set $$
            R_r=\bigvee_{w \in L_r} w(X_1, \cdots , X_n).
        $$
        In particular $R_0=X_1 \vee \cdots \vee X_n$.
        \item For brevity of notation, we write $J=\Omega \Sigma$ in reference to the James construction \cite{james}.
    \end{enumerate}
\end{step}

\begin{step}
We now proceed to construct a map $h:\widetilde\prod_{w\in B} Jw \longrightarrow JR_0.$

For each $r\geq 1$, \Cref{fundamental} provides an equivalence
$$\begin{aligned}
    JR_{r-1} &= J\bigg( X_{r} \vee \: \bigvee_{x_r \ne w\in L_{r-1}} w(X_1, \cdots , X_n) \bigg) \\
    & \simeq JX_{r} \times J \bigg(\: \bigvee_{i\geq 0 \:,\: w \in L_r\setminus \{x_{r}\}}  w(X_1, \cdots , X_n) \wedge X_{r}^{\wedge i} \bigg) \\
    & = JX_{r} \times JR_{r}
\end{aligned}
$$
which we denote $\varphi_r:JX_{r} \times JR_r\longweak JR_{r-1}$. Consider the trivially commuting square $$
\begin{tikzcd}
    * \ar[r] \ar[d] & JX_r \ar[d] \\
    JR_{r-1} & JX_r \times JR_r \ar[l, "\varphi_r"'] 
\end{tikzcd}
$$ 
in which the right vertical map is the inclusion of the first summand. Applying $\prod_{i<r}JX_i \times -$ to this diagram yields the square $$
\begin{tikzcd}
    \prod_{i<r}JX_i \ar[r] \ar[d] & \prod_{i\leq r} JX_i \ar[d] \\
    \prod_{i<r}JX_i \times JR_{r-1} & \prod_{i\leq r}JX_i \times JR_r \ar[l, "\id\times\varphi_r"']
\end{tikzcd}
$$
where the bottom map $(\id\times \varphi_{r})$ is an equivalence since it is a pullback of $\varphi_r$. These squares assemble into the following diagram $$
\label{tower}
\begin{tikzcd}
 \cdots \ar[rr] && \prod_{i<r} JX_i \ar[rr]\ar[d] && \prod_{i\leq r} JX_i \ar[rr]\ar[d] && \cdots \\
\cdots  && \prod_{i<r} JX_i \times JR_{r-1} \ar[ll, "\id\times \varphi_{r-1}"']\ar[d, "\sim" sloped] && \prod_{i\leq r} JX_i \times JR_r \ar[ll, "\id\times \varphi_{r}"']\ar[d, "\sim" sloped] && \cdots \ar[ll, "\id\times \varphi_{r+1}"'] \\
\cdots \ar[rr, equals] &&  JR_0 \ar[rr, equals] && JR_0 \ar[rr, equals] && \cdots 
\end{tikzcd}
$$
in which the bottom vertical maps $\prod_{i\leq r}JX_i \times JR_r \longweak JR_0$ are the composite equivalences $$
(\id\times \varphi_{1})\circ \cdots \circ (\id\times \varphi_{r}).
$$
Disregarding the maps $\id\times \varphi_{r}$ in the diagram above, we obtain a tower whose induced maps on the colimits we denote $$
h=h(X_1, \cdots X_n) : \widetilde\prod_{i\geq 1} JX_i \longrightarrow JR_0.
$$
    
\end{step}

\begin{step}
Now we show that $h$ is $\infty$-connected. To this end, we estimate connectivity of the maps involved in the square 
$$
\begin{tikzcd}
    \prod_{i<r} JX_i \ar[r] \ar[d] & \prod_{i\leq r} JX_i \ar[d] \\
    JR_0 \ar[r, equals] & JR_0.
\end{tikzcd}
$$
Recall that $X_r$ was defined as $x_r(X_1, \cdots X_n)$ where $x_r$ is the $r$-th word in our choice of Hall basis. The number $k_r$ of generators involved in the word $x_r$ tends to $\infty$ with $r$. Recall that $x_r=\min(L_{r-1})$ and that the order on $L_{r-1}$ is compatible with wordlength. Since all the $X_i$ are connected, we have $$
\conn(Jw(X_1, \cdots , X_n)) \geq \conn(JX_r) \geq k_r-1 
$$
for all $w\in L_{r-1}$ by \Cref{connprop}. Then by \Cref{connlemma} the maps $*\to JX_r$ and $* \to JR_r$ are both $(k_r-1)$-connected. By \Cref{connpullback} we find $$
\conn\bigg(\prod_{i<r}JX_i \to \prod_{i<r}JX_i \times JR_{r-1}  \weak JR_0 \bigg) \geq \conn\bigg(\prod_{i<r}JX_i \to \prod_{i\leq r}JX_i\bigg) \geq k_r-1.
$$
At this point, \Cref{inftyconnected} allows us to conclude that $h$ is indeed $\infty$-connected.
\end{step}

\begin{step}
To conclude the proof, we use the following trick from \cite{EHP}. Pick a presentation of $\E$ as an accessible left-exact localization of a presheaf $\infty$-topos, and denote by
\[
\begin{tikzcd}
L: \PSh (\C)_* \ar[r, shift left] &  \E_*\ar[l, shift left, hook] :i
\end{tikzcd}
\]
the localization induced by \Cref{presentability} on categories of pointed objects. By \cite[Remark~6.5.1.15]{HTT} we can find pointed connected objects $X'_i \in \PSh (\C)_*$ so that $X_i \simeq LX'_i$ for $i=1, \cdots, n$. The argument above shows that the map $$
h':=h(X'_1, \cdots , X'_n) : \widetilde\prod_{i\geq 1} JX'_i \longrightarrow JR'_0
$$
is $\infty$-connected in $\PSh(\C)$. But $\PSh(\C)$ is a hypercomplete topos, hence this map is an equivalence. Now from the construction of $h(X'_1, \cdots, X'_n)$ as the map induced on colimits by the tower \eqref{tower} (with $X_i$ replaced by $X'_i$), we see that $$
Lh(X'_1, \cdots, X'_n)\simeq h(LX'_1, \cdots, LX'_n)\simeq h(X_1, \cdots, X_n).
$$
This follows from the fact that $L$ is left-exact, hence commutes with finite products and the functor $J$. But since $h'$ is an equivalence, $h\simeq Lh'$ is also an equivalence and the proof is complete.
\end{step}
\end{proof}


\bibliographystyle{alpha}
\bibliography{refs}

\end{document}